\documentclass[11pt]{article}

\usepackage{geometry}
\usepackage[parfill]{parskip}
\usepackage{graphicx}

\usepackage{amssymb}

\usepackage{epstopdf}

\usepackage{amssymb}
\usepackage{amsmath}
\usepackage{enumerate}
\usepackage{amsthm}

\newfont{\bbb}{msbm10 scaled\magstep1}
\newfont{\bbbsub}{msbm10}

\newcommand{\ba}{\begin{array}}
\newcommand{\ea}{\end{array}}
\newcommand{\be}{\begin{equation}}
\newcommand{\ee}{\end{equation}}

\newcommand{\ds}{\displaystyle}
\newtheorem{theorem}{Theorem}
\newtheorem{lemma}[theorem]{Lemma}

\newtheorem{rem}[theorem]{Remark}
\let\oldrem\rem
\renewcommand{\rem}{\oldrem\normalfont}
\newtheorem{definition}[theorem]{Definition}
\let\olddefinition\definition
\renewcommand{\definition}{\olddefinition\normalfont}

\newcommand{\fd}{\partial}
\newcommand{\eps}{\epsilon}
\newcommand{\Del}{\Delta}

\newcommand{\norm}[1]{\left|#1\right|}
\newcommand{\dnorm}[1]{\left\|#1\right\|}
\newcommand{\ra}{\rightarrow}
\newcommand{\Z}{\mbox{\bbb Z}}

\newcommand{\R}{\mbox{\bbb R}}

\newcommand{\cI}{{\cal I}}

\newcommand{\tx}[1]{\quad\mbox{ #1 }\quad}

\title{A model for studying double exponential growth in the two-dimensional Euler equations}
\author{ Nets Hawk Katz and Andrew Tapay 
   \thanks{Both authors were partially supported by NSF grant DMS 1266104} }

\begin{document}

\maketitle

\begin{abstract}

We introduce a model for the two-dimensional Euler equations which is designed to study whether or not double exponential growth can be achieved at an interior point of the flow.

\end{abstract}

\section{Introduction}

The two-dimensional Euler equations for incompressible fluid flow are given by
$${\partial u \over \partial t} +u \cdot \nabla u  = - \nabla p,$$
together with
$$\nabla \cdot u = 0.$$

Here $u(x,t)$ is a time-varying vector field on ${\Bbb R}^2$ representing the velocity and $p(x,t)$ is a scalar representing the pressure.

The equation is solved with a given initial divergence free velocity field $u_0$:

$$u(x,0)=u_0(x).$$ 

When $u_0$ is chosen to be, for instance, smooth with compact support, a smooth solution to the
Euler equation exists for all time. Moreover, a result of Beale, Kato, and Majda \cite{BKM} shows that
Sobolev norms grow at most double exponentially in time.

Considerable work has been done recently to establish that such growth actually occurs. Denisov \cite{D} demonstrates growth similar to double exponential in an example that consists of a slightly smoothed singuar steady state solution together with a bump. For some time, the singular solution stretches the bump at a double exponential rate. Kiselev and Sverak \cite{KS} do Denisov one better by creating a sustained double exponential growth near a boundary. This is a very similar idea to Denisov's. We may imagine that something quite similar to Denisov's singular steady state lives right at the boundary and is drawing bumps towards it.

The purpose of this paper is to create a tool for studying the question of whether double exponential growth can begin spontaneously at an interior point. We borrow from Pavlovi\'{c}'s thesis \cite{P} the idea that the allowed fast growth in Euler is coming
from low frequency to high frequency interactions. We model the impact of each scale on the vicinity of a given particle as a linear area-preserving map. As each scale evolves, it impacts the effects of the smaller scales. We can model this as a system of differential equations with one $SL(2)$-valued  unknown for each scale. The  main result in the paper is Theorem \ref{main}
below. It says that during a time period, when the sum over scales of the $L^{\infty}$ norm of $\nabla u$ at a single scale is around $N$,
this autonomous system of differential equations closely approximates the actual behavior of the Euler equation for a time period
whose length is of order $O({\log N \over N})$. This is a time period during which growth by a factor of a power of $N$ can occur in
the Sobolev norms of the velocity. Indeed such growth must occur during some such time period for double exponential growth
to take place. Thus our simplified model can be used to study the possibility and likelihood of growth occuring spontaneously
at an interior point. We believe this phenomenon is definitely worthy of more study.

\section{Outline}    

Some notation: let $\psi:\R^2 \ra \R$ be a smooth function such that

$$ \psi(\xi) = \left\{ \begin{array}{ccc}
1  &\text{for} &0<\norm{\xi}<1\\
0  &\text{for} &\norm{\xi}>2
\end{array}
\right.$$

and define the operator $P_0$ to be the Fourier multiplier with symbol $\psi$. Let $\psi_1(\xi) = \psi(\xi/2) - \psi (\xi)$ and for $j>0$, define $P_j$ to be the Fourier multiplier with symbol $\psi_j(\xi):=\psi_1(2^{1-j}\xi)$. For convenience of notation, define $P_j = 0$ for $j<0$. Thus $P_j$ acts like a projection onto the frequency annulus 
$\{\xi:\norm{\xi} \sim 2^j\}$, and $\sum_j P_j$ is the identity because the sum telescopes. These $P_j$ are commonly known as the Littlewood-Paley Operators.
Further, let $\widetilde{P_j} = \sum_{\alpha = -2}^2 P_{j+\alpha}$ and $E_j = \sum_{k<j}P_k$. Note that

$$E_j f(x) = \int f(x+2^{-j}s)\widehat{\psi}(s)ds$$

and $\widehat{\psi}$ is a radial Schwartz function such that $\int\widehat{\psi}=\psi (0)=1$. Hence $E_j$ acts like, and will be referred to as, an averaging operator on scale $\sim 2^{-j}$.

Let $u: \R^2 \times \R \ra \R^2$ be the velocity field of a two-dimensional, 
inviscid, incompressible fluid flow and $\omega = \frac{\fd u_2}{\fd x_1} - \frac{\fd u_1}{\fd x_2}$ the associated 
vorticity. We make some assumptions about $u$ over the time period we will be considering. We will assume that

\begin{equation} \label{state} \sum_{j=0}^{\infty} ||P_j \nabla u||_{L^{\infty}} \lesssim N, \end{equation}

and

\begin{equation} \label{conservation} ||P_j \nabla u||_{L^{\infty}} \lesssim 1 \end{equation}

for all $j\geq 0$. We define the flow maps $\phi(x,t)$ to be solutions of the differential equations
 
\begin{equation} \label{phi}
\begin{aligned}
\frac{\partial}{\partial t} \phi(x,t) &= u(\phi(x,t),t)\\
\phi(x,0) 
&= x
\end{aligned}
\end{equation}

and so the point $\phi(x,t)$ is the image of the point $x$ under the flow with velocity field $u$ at time $t$. 
Thus, the Jacobian matrix of $\phi$, which we denote by $D\phi$, satisfies the differential equation

\begin{equation} \label{Dphi}
\begin{aligned}
\frac{\partial}{\partial t} D\phi(x,t) 
&= ((\nabla u) \circ \phi) (x,t) \cdot D\phi(x,t)\\
D\phi(x,0) &= I
\end{aligned}
\end{equation}

for each $x \in \R^2$. By both $D$ and $\nabla$, we denote the Jacobian derivative in the spatial variable $x$, and not in the coordinates of the particle trajectories $\phi(x,t)$. Indeed, it should be noted that the equations (\ref{phi}) and (\ref{Dphi}) invite a change of coordinates via the map $x \mapsto \phi(x,t)$. This change of coordinates is especially convenient because incompressibility, $\nabla_x \cdot u = 0$, gives $\det (D_x \phi(x,t)) \equiv 1$. This will  make it useful for our purposes to use the Lagrangian reference frame, that is, spatial variables will be evaluated along the particle trajectories $\phi(x,t)$. A thorough discussion of particle trajectory maps and the Lagrangian reference frame can be found in \cite{BM}.\\

Proceeding formally, if we define $R_i := \Del^{-\frac{1}{2}}\frac{\fd}{\fd x_i}$, we have the so-called 
Biot-Savart Law:

$$\nabla u = \left(\ba{cc}
 -R_1R_2 \omega & -R_2^2 \omega \\
 R_1^2 \omega & R_1R_2 \omega
 \ea\right)$$

Using the the Green's function for the Laplace operator, we can calculate the nonlocal parts of the composed
Riesz operators by giving the non-singular part of their kernels. (The local part, of course, lives
in the singular part of the kernel located on the diagonal.)

$$R_1R_2\omega = K_{12}*\omega(\cdot,t), \quad R_1^2\omega = K_{11}*\omega(\cdot,t), 
\tx{and} R_2^2=-K_{11}*\omega(\cdot,t)$$
 
where

$$K_{12}(x_1, x_2) =\frac{x_1x_2}{\pi(x_1^2+x_2^2)^2} \tx{and} K_{11}(x_1, x_2) = 
\frac{(x_2^2-x_1^2)}{2\pi(x_1^2+x_2^2)^2}.$$

The following definition is the one of the main fixtures of this paper. We will define an approximation of $\nabla u (\phi(0,t),t)$ so that, for a short time, the flow is given by a linear map at each physical scale around the point $\phi(0,t)$. That is, the contribution to $\nabla u(\phi(0,t),t)$ from the part of the vorticity which at time 0 was at an annulus at scale $2^{j}$ around 0 is calculated as though
 the flow on the annulus was  linear given by some $h \in SL(2)$. This is inspired by the following version of the Biot-Savart Law:

\be\label{BS}\begin{split}\nabla u (\phi(0,t),t) &= \int \omega (s,t) K(s - \phi(0,t))ds\\
& = \int \omega (\phi(s,t),t) K(\phi(s,t) - \phi(0,t))ds\\
&= \sum_{j\in \Z} \int_{A_j} \omega (\phi(s,t),t) K(\phi(s,t) - \phi(0,t))ds
\end{split}\ee

where $A_j = \{x:2^{-j}\leq\norm{x}<2^{1-j}\}$ and by dropping the index of $K$ we mean a generic entry in the matrix $\nabla u (\phi(x,t),t)$. We have used the aforementioned change of coordinates $s \mapsto \phi(s,t)$. Here we are focusing on $\phi(0,t)$, which we think of as a generic interior point of a fluid flow in order to study whether double exponential growth is in the making at that point as it moves along the flow. 

%Definition

\begin{definition}
 Let $h(t)$ be an element of $SL(2)$.
 
$$(\nabla u)_{j,h} (t)= \left(\ba{cc}
-(\nabla u)_{j,h,2} & (\nabla u)_{j,h,1} \\
(\nabla u)_{j,h,1} & 
(\nabla u)_{j,h,2} \\
\ea\right)$$

where

$$(\nabla u)_{j,h,i} (t) = \int_{A_j} \omega_{0,j}(s) K_{1i}(h(t) \cdot s)ds,$$

$$\omega_j = \chi_{A_j}(E_{j+\log N}-E_{j-\log N}) \omega,$$

and

$$\omega_{0,j}(x) = \omega_j(x,0).$$

\end{definition}

\begin{rem}The purpose of using $\omega_{0,j}$ instead of just $\omega_0$ is a technical advantage: $\omega_{0,j}$ is a projection onto the frequencies of $\omega$ that make a significant contribution to $\nabla u(\phi(0,t),t)$ coming from the annulus $A_j$. Indeed, if we define

$$\widetilde{\nabla u} (\phi(0,t),t) := \sum_{j\in \Z} \int_{A_j} \omega_j (\phi(s,t),t) K(\phi(s,t) - \phi(0,t))ds $$

whereas (at least formally)

$$\nabla u (\phi(0,t),t) = \sum_{j\in \Z} \int_{A_j} \omega (\phi(s,t),t) K(\phi(s,t) - \phi(0,t))ds, $$

the difference is

\begin{align*}\sum_j \sum_{\norm{k-j}>\log N} \int_{A_j} &P_k( \omega(\phi(s,t),t))K(\phi(s,t) - \phi(0,t))ds\\
&= \sum_{\{(j,k) : \norm{k-j}>\log N\}} \int_{A_j} \left(\int \omega(y)\check{\psi_k}(y-s)dy\right)K(\phi(s,t) - \phi(0,t))ds\\
& = \sum_{\{(j,k) : \norm{k-j}>\log N\}} \int_{A_j}  \left(\int \omega(y)\left(\int e^{2\pi i (y-s)\cdot \xi}\psi_k(\xi)d\xi\right) dy\right)K(\phi(s,t) - \phi(0,t))ds.
\end{align*}

We integrate by parts in $\int e^{2\pi i (y-s)\cdot \xi}\psi_k(\xi)d\xi$, moving a derivative $\nabla_\xi$ from the exponential onto $\psi_k$ for terms in which $k>j+\log N$, and the opposite way for terms where $k<j+\log N$. Since $\norm{s}\sim 2^{-j}$ and we obtain a factor of $2^{\pm k}$ from the dilation of $\psi_1$, this gives the estimate

$$ \lesssim \sum_{\{(j,k) : \norm{k-j}>\log N\}} 2^{-\norm{k-j}} \dnorm{P_k \omega}_{L^\infty} \int_{A_j}K(\phi(s,t)-\phi(0,t))ds.$$

We will only be considering a time period of order $(\log N)/N$. This, together with (\ref{phi}), the Fundamental Theorem of Calculus and (\ref{state}), gives us $\norm{\phi(s,t)-\phi(0,t)} \gtrsim \norm{s}/\log N$. Since, by definition, $\norm{K(x)}\sim \norm{x}^{-2}$, we have $\int_{A_j}K(\phi(s,t)-\phi(0,t))ds \lesssim (\log N)^2$ for all $j$. Hence, we now have the bound

   \begin{align} 
&  \lesssim (\log N)^2 \sum_{\{(j,k) : \norm{k-j}>\log N\}} \dnorm{P_k \omega}_{L^\infty} 2^{-\norm{k-j}}\notag \\
& \lesssim  \frac{(\log N)^2}{N} \sum_{k} \dnorm{P_k \omega}_{L^\infty} \notag \\
& \lesssim (\log N)^2. \label{above}
\end{align} 

We also used (\ref{state}) and the fact that $\dnorm{P_k \omega}_{L^\infty} \sim \dnorm{P_k \nabla u}_{L^\infty} $. The technical advantage of using $\omega_{0,j}$ is that we have, similarly to (\ref{state}),

\be\label{state2} \sum_{j=0}^\infty \dnorm{\omega_j}_{L^\infty} \lesssim \sum_{j=0}^\infty \sum_{\norm{k-j}<\log N} \dnorm{P_k \omega}_{L^\infty} \lesssim N \log N.\ee

The reader might ask why we chose to have the error estimate in (\ref{above}) come to $(\log N)^2$. It is entirely arbitrary. By replacing
the range of $\log N$  scales by a range of $C \log N$ scales which would only cost us a constant in the estimate (\ref{state2}),
we could reduce the estimate to an arbitrary negative power of $N$, but the point is that because of the brevity of our time
period, any estimate for the error which has a power of $N$ lower than 1 will work. The error is smaller than the worst
case we have for $||\nabla u||_{L^\infty}$.
The important part of these estimates is that we lose (at most) a factor of a power of $\log N$ in (\ref{state2}), which is enough for our purposes, mainly because of the assumption (\ref{state}).
\end{rem}

We now state the main result: for a short time, we can approximate the average of the Jacobian of the flow map at the scale $2^{-j}$ by a linear map for each $j$ and these linear maps satisfy an autonomous system of differential equations not depending on the
solution to the Euler equations. The behavior of this system can be a test for whether double exponential growth can occur and
what it should look like.\\

%Theorem main, the main theorem

\begin{theorem} \label{main}
Let $h_j \in SL(2)$ be defined as the solution to 
the ODE

\begin{align*}
\frac{dh_j}{dt} &= \left( \sum_{k<j} (\nabla u)_{k,h_k} \right)h_j\\
h_j(0) &= I
\end{align*}

Then there is a universal constant $C$ such that, for times $t \leq C(\log N)/N$, we have

$$\norm{h_j - E_j D\phi} = O(N^{-\frac{7}{10}})$$

for all $j>0$.
\end{theorem}

Many of the estimates will be based on the following Lemma, which says that solutions to similar ODEs remain similar for a short time.\\

%Lemma ODE, basic ODE lemma

\begin{lemma} \label{ODE}
Suppose that $F, G_1, G_2, w$ and $v$ are functions of time such that $F(t) = O(N)$, and that 

\begin{align*}
\frac{dw}{dt}(t) &= F(t)w(t) + G_1(t)\\
\frac{dv}{dt}(t) &= F(t)v(t) + G_2(t)\\
w(0) &= v(0).
\end{align*}

Assume further that, for some constant $E$,

$$\norm{G_1(t)-G_2(t)} \lesssim \left\{ \begin{array}{ccc}
E  &\text{for} &0<t\lesssim1/N\\
\norm{F(t)(w(t)-v(t))}  &\text{for} &t\gtrsim1/N
\end{array}
\right.$$

Then, there is a universal constant $C$, independent of $N$, such that $\norm{(w-v)(t)} \lesssim EN^{-\frac{9}{10}-\frac{1}{100}}$ for all times $t \leq C(\log N)/N$.
\end{lemma}

The idea behind Lemma \ref{ODE} is that, since the difference starts out at 0, the ``error" term $G_1-G_2$ dominates for times $t \lesssim 1/N$. At that time, the main term $F(t)(w(t)-v(t))$ becomes the dominant term but $\norm{(w-v)(t)}$ remains relatively small for an additional time $\lesssim \log N$. Most of the time, we won't need the extra factor of $N^{-\frac{1}{100}}$. It will be used to eliminate factors of $\log N$ that show up in the error term $E$.\\

A straightforward but necessary application of Lemma \ref{ODE} is\\

%Lemma PjDphi

\begin{lemma} \label{PjDphi}

Under the assumptions (\ref{state}) and (\ref{conservation}), $\sup_j \dnorm{P_j D\phi(t)}_{L^\infty} \lesssim N^{-\frac{9}{10}-\frac{1}{100}}$ for times $t \lesssim C(\log N)/N$.

\end{lemma}

In order to prove Theorem \ref{main}, we will show that, for the time period we are considering, the flow maps are essentially linear on a given dyadic annulus. That is, we will estimate the difference between the linear map $E_j D\phi(0,t) \cdot x$ and the difference $\phi(x,t) - \phi(0,t)$ for $\norm{x} \sim 2^{-j}$. To do so, we first show that the averages of the Jacobians of the flow maps are close to the averages of the differences in the flow maps, that is

$$E_j D\phi (0,t) \cdot x - \big( E_j \phi (x,t) - E_j \phi(0,t) \big) \lesssim 2^{-j} N^{-\frac{9}{10}},$$

a sort of approximate Mean Value Theorem. We do this by using (\ref{phi}) to examine the time derivative of the difference of the flow maps, and (\ref{Dphi}) to examine the average of $D\phi$ at the appropriate scale. With the Fundamental Theorem of Calculus and on frequency support grounds, we have that the time derivative of the difference is essentially 

$$\left( \int_0^1 E_j\nabla u (s\phi(x,t)+(1-s)\phi(0,t),t)ds \right) \cdot (E_j\phi(x,t) - E_j\phi(0,t)).$$

If we throw away $\log N$ many frequencies from the integrand, it is almost constant on its domain. The error from doing so is acceptable, and so we have now essentially

$$E_j \nabla u (0,t) \cdot (E_j\phi(x,t) - E_j\phi(0,t)) + O(2^{-j}\log N)$$

and we apply Lemma \ref{ODE}. We will still have to show that the difference of averages is close to the actual difference for $x$ at the appropriate scale. Since $\sum P_k = 1$, this is entirely a matter of controlling the frequencies bands bigger than $2^{j}$. We do this by first using a trivial bound for the high ($\geq j + \log N$) frequencies, which comes from the Fundamental Theorem of Calculus. For frequencies $j \leq k \leq j+\log N$, we can again exploit the fact that averages at scale $2^{-j}$ are essentially constant at scale $2^{-j - \log N}$.\\

Putting all of this together, we have\\

%Lemma diff, the averages of the Jacobians of the flow maps are close to the differences of the flow maps

\begin{lemma} \label{diff}
  For times $t \leq C(\log N)/N$ and $\norm{x} \sim 2^{-j}$, we have

$$\norm{E_j D\phi(0,t)\cdot x - (\phi(x,t) - \phi(0,t))} = O( 2^{-j} N^{-\frac{9}{10}}).$$
\end{lemma}

Finally, we will prove Theorem \ref{main} by using Lemma \ref{diff} to substitute the linear map $E_jD\phi(0,t)\cdot x$ for the difference $\phi(x,t)-\phi(0,t)$ in each piece of the convolution used to calculate $\nabla u$ by the Biot-Savart Law.

\section{Proofs}

%Proof of Lemma ODE

We restate Lemma \ref{ODE}: Suppose that $F, G_1, G_2, w$ and $v$ are functions of time such that $F(t) = O(N)$, and that 

\begin{align*}
\frac{dw}{dt}(t) &= F(t)w(t) + G_1(t)\\
\frac{dv}{dt}(t) &= F(t)v(t) + G_2(t)\\
w(0) &= v(0).
\end{align*}

Assume further that, for some constant $E$,

$$\norm{G_1(t)-G_2(t)} \lesssim \left\{ \begin{array}{ccc}
E  &\text{for} &0<t\lesssim1/N\\
\norm{F(t)(w(t)-v(t))}  &\text{for} &t\gtrsim1/N
\end{array}
\right.$$

Then, there is a universal constant $C$, independent of $N$, such that $\norm{(w-v)(t)} \lesssim EN^{-\frac{9}{10}-\frac{1}{100}}$ for all times $t \leq C(\log N)/N$.

\begin{proof}[Proof of Lemma \ref{ODE}:] Observe that

\begin{align*}
\frac{d(w-v)}{dt}(t) &= F(t)(w-v)(t) - (G_1(t) - G_2(t)) \\
(w-v)(0) &= 0
\end{align*}

and suppose that $T$ is the first time that such that $\norm{(w-v)(T)}=E/N$. Then, for times $t \leq \min \{ \frac{1}{N},T \}$, we have by assumption

$$F(t)(w-v)(t) = O(E) \implies \norm{\frac{d(w-v)}{dt}(t)} \lesssim E.$$

Therefore, since the growth of the difference is at most linear of rate $E$, it follows that $T = O(\frac{1}{N})$. For $T \leq t \leq C(\log N)/N$, we have

$$\norm{\frac{d(w-v)}{dt}(t)} \lesssim \norm{F(t)(w-v)(t)} = O(N)\norm{(w-v)(t)}.$$

and so by Gronwall's lemma we have

$$\norm{(w-v)(t)} \lesssim \frac{E}{N}e^{Nt} \lesssim EN^{-\frac{9}{10}-\frac{1}{100}}$$

where we get the last inequality by choosing $C$ such that $t\lesssim C(\log N)/N \leq (\log (N^{\frac{1}{10} - \frac{1}{100}}))/N$. 

\end{proof}

Recall Lemma \ref{PjDphi}: Under the assumptions (\ref{state}) and (\ref{conservation}), $\sup_j \dnorm{P_j D\phi(t)}_{L^\infty} \lesssim N^{-\frac{9}{10}-\frac{1}{100}}$ for times $t \lesssim C(\log N)/N$.

%Proof of Lemma PjDphi

\begin{proof}[Proof of Lemma \ref{PjDphi}]: Taking $P_j$ of both sides of (\ref{Dphi}), we have, on frequency support grounds

$$\frac{\fd}{\fd t}P_j D\phi = P_j (E_j (\nabla u \circ \phi) \cdot \widetilde{P_j} D\phi) + P_j (\widetilde{P_j} (\nabla u \circ \phi) \cdot E_j D\phi) + P_j \left(\sum_{k>j} \widetilde{P_k} (\nabla u \circ \phi) \cdot \widetilde{P_k} D\phi\right).$$

We will make frequent use of the following versions of the cheap Littlewood-Paley inequality:

$$\sup_j \dnorm{P_j f}_{L^\infty} \lesssim \dnorm{f}_{L^\infty}\tx{and} \sup_j \dnorm{E_j f}_{L^\infty} \lesssim \dnorm{f}_{L^\infty}.$$

To prove this, observe that, for example,

$$\norm{E_j f(x)} = \norm{\int f(x+2^{-j}s)\widehat{\psi}(s)ds} \leq \dnorm{f}_{L^\infty} \norm{ \int\widehat{\psi}(s)ds} \lesssim \dnorm{f}_{L^\infty} $$

by the definition of $\psi$. Let $S(t) = \sup_j \dnorm{P_j D\phi(t)}_{L^\infty}$. Then we have

$$P_j (E_j (\nabla u \circ \phi) \cdot \widetilde{P_j} D\phi) = O(N)S(t)$$

because we can use the cheap Littlewood-Paley inequality to drop the $P_j$, and use it again together with (\ref{state}) on the $\nabla u$ term to obtain a factor of $O(N)$. By the definition of $\widetilde{P_j}$ and the cheap Littlewood-Paley inequality, we have the factor of $O(1)S(t)$. Along similar lines, using (\ref{conservation}) and the cheap Littlewood-Paley inequality, we have

$$P_j(\widetilde{P_j} (\nabla u \circ \phi) \cdot E_j D\phi) = O(\dnorm{D\phi}_{L^\infty})$$

and finally

$$P_j \left(\sum_{k>j} \widetilde{P_k} (\nabla u \circ \phi) \cdot \widetilde{P_k} D\phi\right) = O(N)S
(t)$$

again by (\ref{state}) and the cheap Littlewood-Paley inequality. Putting this together, we have

\be\frac{d}{dt}S(t) = O(N) S(t) + O( \dnorm{D\phi}_{L^\infty}).\label{Sode}\ee

Using (\ref{Dphi}), (\ref{state}) and Gronwall's Lemma, we see that

\be\label{Dphismall}
\dnorm{D\phi}_{L^\infty} \lesssim e^{Nt}
\ee

and so the the second term of (\ref{Sode}) is $O(1)$ for times $t \lesssim 1/N$ and is dominated by the first term for times $t\gtrsim 1/N$. With $w = S$ and $v(t) = S(0) = O(1)$, we may now apply Lemma \ref{ODE}, and this finishes the proof.

\end{proof}

%Proof of Lemma diff

Recall Lemma \ref{diff}: For times $t \leq C(\log N)/N$ and $\norm{x} \sim 2^{-j}$, we have

$$\norm{E_j D\phi(0,t)\cdot x - (\phi(x,t) - \phi(0,t))} = O( 2^{-j} N^{-\frac{9}{10}}).$$

The proof is achieved in two parts. First, we show that the average of the Jacobian of a flow map is closely approximated by the average difference of a flow map at a fixed scale. That is, for $\norm{x} \sim 2^{-j}$, we have 

$$\norm{E_j D\phi (0,t) \cdot x - \big( E_j \phi (x,t) - E_j D\phi (0,t) \big)} = O(2^{-j}N^{-\frac{9}{10}}).$$

We do this by comparing the time derivatives of each expression and using Lemma \ref{ODE}. Then, we show that the differences of the flow maps themselves at scale $\sim 2^{-j}$ are closely approximated by their averages at the same scale, that is

$$\norm{E_j \phi(x,t) - E_j \phi(0,t) -(\phi(x,t) - \phi(0,t))} \lesssim 2^{-j} N^{-\frac{9}{10}}$$

hence proving Lemma \ref{diff} by the triangle inequality.

\begin{proof}[Proof of Lemma \ref{diff}:]

First, we claim that, for $\norm{x}\sim 2^{-j}$,

\be \label{barrier}
\norm{E_j D\phi (0,t) \cdot x - \big( E_j \phi (x,t) - E_j D\phi (0,t) \big)} = O(2^{-j}N^{-\frac{9}{10}}).
\ee

We examine $\frac{\fd}{\fd t} E_j D\phi (0,t) \cdot x$ using (\ref{Dphi}). The goal is to use Lemma \ref{ODE}. In this case, we want to show that  $\frac{\fd}{\fd t} E_j D\phi (0,t) = E_j \nabla u \circ \phi(0,t) \cdot E_j D\phi (0,t) \cdot x$ plus an error term which obeys acceptable bounds. Taking $\cdot x$ and $E_j$ of both sides of (\ref{Dphi}), we have, purely on frequency support grounds

\be \label{hihi}
\begin{aligned}\frac{\fd}{\fd t} E_j D\phi (0,t) \cdot x &= E_j \Big(E_j (\nabla u \circ \phi) (0,t) \cdot E_j D\phi (0,t) \cdot x\Big)\\
&+ E_j\left( \sum_{l>j} \widetilde{P_l} (\nabla u \circ \phi) (0,t) \cdot \widetilde{P_l} D\phi (0,t)\cdot x\right).
\end{aligned}
\ee

For the first term, since $E_j$ is not actually a projection, we have to separate some of the frequencies. We use the fact that $E_j E_{j-4} = E_{j-4}$, and we have

\begin{align*}E_j \Big(E_j (\nabla u \circ \phi) (0,t) \cdot E_j D\phi (0,t) \cdot x\Big)& = E_{j-4} (\nabla u \circ \phi) (0,t) \cdot E_{j-4} D\phi (0,t)\cdot x\\
& + E_j \left( \sum_{k,l = j-4}^{j-1} P_k (\nabla u \circ \phi) (0,t) \cdot P_l D\phi (0,t)\cdot x \right).
\end{align*}

We now add and subtract $\sum_{k,l = j-4}^{j-1} P_k (\nabla u \circ \phi) (0,t) \cdot P_l D\phi (0,t)\cdot x $. This gives us

\begin{align*} \frac{\fd}{\fd t} &E_j D\phi (0,t) \cdot x = E_j \nabla u \circ \phi(0,t) \cdot E_j D\phi (0,t) \cdot x\\
&+  \sum_{k,l = j-4}^{j-1} E_j \Big(P_k (\nabla u \circ \phi) (0,t) \cdot P_l D\phi (0,t)\cdot x \Big) -  P_k (\nabla u \circ \phi) (0,t) \cdot P_l D\phi (0,t)\cdot x
\end{align*}

where we think of the sum as being an error term. There are only $O(1)$ terms in the sum. We can drop the $E_j$ using the cheap Littlewood-Paley inequality, and hence we only need to bound a typical term in the sum, i.e. $P_k (\nabla u \circ \phi) (0,t) \cdot P_l D\phi (0,t)\cdot x$. On this, we use (\ref{conservation}) on the $\nabla u$ term and Lemma \ref{PjDphi} on the $D\phi$ term, and hence the sum is bounded by $O(2^{-j})$ because of the factor of $\cdot x$, and we have achieved the goal

\be \label{dtDphi} \frac{\fd}{\fd t} E_j D\phi (0,t) \cdot x = E_j (\nabla u \circ \phi)(0,t) \cdot E_j D\phi (0,t) \cdot x + O(2^{-j}). \ee

To use Lemma \ref{ODE}, we now need the analogous statement for $\frac{\fd}{\fd t} \big( E_j \phi (x,t) - E_j D\phi (0,t) \big) $. We begin by using (\ref{phi}). Since $\frac{\fd}{\fd t}$ commutes with $E_j$, and by (\ref{phi}) and the Fundamental Theorem of Calculus, we have

\begin{align*}\frac{\fd}{\fd t} \big( E_j \phi (x,t) - E_j \phi (0,t) \big) &= E_j \frac{\fd}{\fd t} \big( \phi(x,t) - \phi(0,t) \big)\\
&= E_j \big( u(\phi(x,t),t) - u(\phi(0,t),t) \big) \\
&= E_j\left[ \left( \int_0^1 \nabla u (s\phi(x,t)+(1-s)\phi(0,t),t)ds \right) \cdot (\phi(x,t) - \phi(0,t)) \right].
\end{align*}

We now take $E_j$ of the product, move $E_j$ inside the integral and the above expression gives

\be
\begin{array}{c} \label{hihi2} \ds{E_j\left(\left[ \int_0^1 E_j \nabla u (s\phi(x,t)+(1-s)\phi(0,t),t)ds\right]\cdot \big( E_j\phi(x,t) - E_j\phi(0,t) \big) \right)}\\
\ds{+ E_j \left(\sum_{l>j} \widetilde{P_l} (\nabla u \circ \phi) \cdot \widetilde{P_l}\big(\phi(x,t) - \phi(0,t)\big) \right) }
\end{array}
\ee

which we justify on frequency support grounds. We use the same technique on the first term as we used to achieve (\ref{dtDphi}). That is, we add and subtract

$$\sum_{k,l = j-4}^{j-1} \left[ \int_0^1 P_k \nabla u (s\phi(x,t)+(1-s)\phi(0,t),t)ds\right] \cdot (P_l\phi(x,t) - P_l \phi(0,t)) $$

to exploit the fact that $E_jE_{j-4} = E_{j-4}$. This gives us

\begin{align} (\ref{hihi2})& = \left[ \int_0^1 E_j \nabla u (s\phi(x,t)+(1-s)\phi(0,t),t)ds\right]\cdot \big( E_j\phi(x,t) - E_j\phi(0,t) \big) \label{good}\\
&+ \sum_{k,l = j-4}^{j-1} E_j \left( \left[ \int_0^1 P_k \nabla u (s\phi(x,t)+(1-s)\phi(0,t),t)ds\right] \cdot (P_l\phi(x,t) - P_l \phi(0,t)) \right) \label{bad1}\\
&- \sum_{k,l = j-4}^{j-1} \left[ \int_0^1 P_k \nabla u (s\phi(x,t)+(1-s)\phi(0,t),t)ds\right] \cdot (P_l\phi(x,t) - P_l \phi(0,t))\label{bad2}\\
&+ E_j \left(\sum_{l>j} \widetilde{P_l} (\nabla u \circ \phi) \cdot \widetilde{P_l}\big(\phi(x,t) - \phi(0,t)\big)\right) \label{bad3}
\end{align}

The term (\ref{good}) is good and the terms (\ref{bad1}), (\ref{bad2}) and (\ref{bad3}) are error terms. Note that (\ref{phi}), (\ref{state}) and Gronwall's Lemma give us

\be \label{phismall} \norm{\phi(x,t) - \phi(0,t)} \lesssim \norm{x}\log N \ee

for times $t\leq C (\log N)/N$. We can use the cheap Littlewood-Paley inequality to drop the $E_j$ in (\ref{bad1}) and so we may combine it conveniently with (\ref{bad2}). Since there are only $O(1)$ terms in the sum, we only need to bound one of them. We use (\ref{conservation}) on the integrand, and the cheap Littlewood-Paley inequality together with (\ref{phismall}) on the difference $P_l\phi(x,t) - P_l \phi(0,t)$. This gives us 

$$(\ref{bad1}) + (\ref{bad2}) = O(2^{-j} \log N).$$

For (\ref{bad3}), we can use the cheap Littlewood-Paley inequality and (\ref{conservation}) to drop the factor of $\widetilde{P_l} (\nabla u \circ \phi)$. We then use (\ref{phismall}) and the cheap Littlewood-Paley inequality and we have

$$(\ref{bad3}) = O(2^{-j} \log N).$$

Putting the last two equations together gives

\be(\ref{hihi2}) = \left[ \int_0^1 E_j \nabla u (s\phi(x,t)+(1-s)\phi(0,t),t)ds\right]\cdot \big( E_j\phi(x,t) - E_j\phi(0,t) \big) + O(2^{-j} \log N).\label{dtdiff}\ee

At this point, we reiterate that the goal is to have the above expression equal to $E_j (\nabla u \circ \phi)(0,t) \big( E_j\phi(x,t) - E_j\phi(0,t) \big)$ plus an acceptable error term, and that the error so far, $O(2^{-j} \log N)$, is acceptable. For convenience, we adopt the following notation for the integral term in (\ref{dtdiff}):

\begin{align*} \cI (t) &:= \int_0^1 E_j \nabla u (s\phi(x,t)+(1-s)\phi(0,t),t)ds \\
& = \int_0^1 \Big(E_k \nabla u (s\phi(x,t)+(1-s)\phi(0,t),t) + \sum_{l=k}^{j-1} P_l \nabla u (s\phi(x,t)+(1-s)\phi(0,t),t)\Big) ds
\end{align*}

where we chose $k = j -\log N$ so that the first part of the integral is essentially constant. Indeed, if $\dnorm{f}_{L^\infty} \lesssim N$ and $\norm{x-y} \leq 2^{-j}$, with this choice of $k$ we have

\begin{align*} E_k f(x) - E_k f(y) &= \int_{\R^2} f(s) 2^{2k} \Big(\widehat{\psi}(2^k(x+s)) - \widehat{\psi}(2^k(y+s))\Big)ds\\
&\lesssim 2^{2k} \dnorm{f}_{L^\infty} \dnorm{\nabla \widehat{\psi}}_{L^\infty}2^k \norm{x-y} \norm{B_{2^{-j}}(0)}\\
& \lesssim \dnorm{f}_{L^\infty} 2^{k-j}\\
& \lesssim N 2^{-\log N}\\
& \lesssim 1
\end{align*}

wherein we can move from the first line to the second line by the definition of $\psi$, more specifically the fact that $\psi \leq 1$ and is supported on $B_2(0)$. Since the first part of integrand is essentially constant, we can choose any point in the domain we want for its argument (we choose $\phi(0,t)$). We then add and subtract the extra frequencies (that is, those between $k$ and $j$), and we have

$$\cI (t) = E_j \nabla u (\phi(0,t),t) + \int_0^1 \Big(\sum_{l=k}^j P_l \nabla u (s\phi(x,t)+(1-s)\phi(0,t),t) - P_l \nabla u (\phi(0,t),t)\Big) ds  +O(1).$$

The integral of the sum is trivially $\lesssim \log N$ because of (\ref{conservation}) and the choice of $k$. Substituting this into (\ref{dtdiff}) and, again, using the fact that $\norm{E_j\phi(x,t) - E_j\phi(0,t)} = O(2^{-j}\log N)$, we have (finally)

$$\frac{\fd}{\fd t} \big( E_j \phi (x,t) - E_j \phi (0,t) \big)  = E_j \nabla u (\phi(0,t),t) \cdot \big( E_j \phi (x,t) - E_j \phi (0,t) \big) + O(2^{-j}(\log N)^2).$$

Using this, together with (\ref{dtDphi}), we can apply Lemma \ref{ODE} with $w = E_jD\phi(0,t)\cdot x$, $v = E_j\phi(x,t) - E_j\phi(0,t)$ and $E= 2^{-j} (\log N)^2$, which proves the claim that

$$\norm{E_j D\phi(0,t)\cdot x - (\phi(x,t) - \phi(0,t))} = O( 2^{-j} N^{-\frac{9}{10}})$$

for $\norm{x}\sim 2^{-j}.$\\

Now, to prove the Lemma, it suffices to show that 

$$\norm{E_j \phi(x,t) - E_j \phi(0,t) -(\phi(x,t) - \phi(0,t))} \lesssim 2^{-j} N^{-\frac{9}{10}}.$$

By the definition of the Littlewood-Paley Operators, we have

$$E_j \phi(x,t) - E_j \phi(0,t) -(\phi(x,t) - \phi(0,t))= \sum_{k\geq j} P_k \phi(x,t) - P_k \phi(0,t)$$

which we now estimate in two parts. First, for the large frequencies, we have (for arbitrary $y$)

\be\begin{split}\sum_{k = j+\log N}^\infty P_k \phi(y,t) &= \sum_{k = j+\log N}^\infty E_{k+1} \phi(y,t) - E_k\phi(y,t)\\
&= \sum_{k = j+\log N}^\infty \int_{\R^2} \big(\phi(y+2^{-(k+1)}s,t) - \phi(y+2^{-k}s,t)\widehat{\psi}(s)ds\\
&\lesssim \dnorm{D\phi}_{L^\infty} \sum_{k = j+\log N}^\infty 2^{-k}\\
&\lesssim N^\frac{1}{10} 2^{-j-\log N}\\
& \lesssim 2^{-j} N^{-\frac{9}{10}} 
\end{split}\label{largefreq}\ee

where we have used the definition of $E_k$, (\ref{Dphismall}) and our choice of $C$ (as in the proof of Lemma \ref{ODE}). For the smaller frequencies, we have left

\be \sum_{k = j}^{l} P_k \phi(x,t) - P_k \phi(0,t) \label{sum}\ee

where $l = j+\log N -1 $. We will estimate an arbitrary frequency band $P_k \phi(x,t) - P_k \phi(0,t)$ in this rage. Take $x_i$ to be points on the line segment from $0$ to $x$ such that $\norm{x_{i+1}-x_i}\sim 2^{-l}$, thus we have $\sim 2^{l-j} \sim N$ points $x_i$. For convenience of notation, take $x_0=0$ and $x_N=x$. By adding and subtracting $P_k \phi(x_i,t)$ for each $i$, we have

\be \norm{P_k \phi(x,t) - P_k \phi(0,t)} \lesssim 2^{l-j} \max_i \norm{P_k \phi(x_{i+1},t) - P_k \phi(x_i,t)}.\label{max}\ee

For each $i$, we have from Lemma \ref{PjDphi}

$$P_k \big(\phi(x_{i+1},t) - \phi(x_i,t)\big) \lesssim 2^{-l} \dnorm{P_k D\phi}_{L^\infty} \lesssim 2^{-l} N^{-\frac{9}{10}-\frac{1}{100}} $$

Plugging this into (\ref{max}), and, in turn plugging the result into (\ref{sum}), we can use the factor of $N^{-\frac{1}{100}}$ and the fact that there are only $\sim \log N$ terms in the sum to obtain

$$\sum_{k = j}^{l} P_k \phi(x,t) - P_k \phi(0,t) \lesssim 2^{-j}N^{-\frac{9}{10}}.$$

This, together with (\ref{largefreq}), proves the claim that

$$\norm{E_j \phi(x,t) - E_j \phi(0,t) -(\phi(x,t) - \phi(0,t))} \lesssim 2^{-j} N^{-\frac{9}{10}} $$

and we have already shown that 

$$\norm{E_j D\phi(0,t)\cdot x - (\phi(x,t) - \phi(0,t))} = O( 2^{-j} N^{-\frac{9}{10}})$$

and applying the triangle inequality we complete the proof of Lemma \ref{diff}.

\end{proof}

%Proof of Theorem main
  
\begin{proof}[Proof of Theorem \ref{main}:]

Our goal is to show that

\be \label{bigode} \frac{d(E_j D\phi - h_j)}{dt} = \left( \sum_{k<j} (\nabla u)_{k,E_k D\phi,i} \right) E_j D\phi - \left( \sum_{k<j} (\nabla u)_{k,h_k,i} \right) h_j + O(N^\frac{1}{5})E_j D\phi \ee

and apply a version of Lemma \ref{ODE}. (We will denote $E_jD\phi(0,t)$ by $E_jD\phi$ in order to simplify notation.) We first have to estimate

$$\norm{E_j \widetilde{\nabla u} - \sum_{k<j} (\nabla u)_{k,E_k D\phi}}$$

where, by dropping the index $i$, we mean a generic entry of the matrix. Also, we are free to use $\widetilde{\nabla u}$ in stead of $\nabla u$, since the two differ by only $O((\log N)^2)$ and the error term in (\ref{bigode}) will dominate this. We need to estimate the above difference because because it is the error term between 

$$\left( \sum_{k<j} (\nabla u)_{k,E_k D\phi} \right) E_j D\phi $$

and (\ref{dtDphi}), the ODE that $E_jD\phi$ actually obeys (if one omits the $\cdot x$ from (\ref{dtDphi}).) We are, indeed, ignoring the error terms from (\ref{dtDphi}), but they are controlled by the error term in (\ref{bigode}). Hence, we have to estimate

\be E_j \left( \sum_{k \in \Z} \int_{A_k} \omega_{0,k}(s) K(\phi(s,t)-\phi(0,t))ds \right) - \sum_{k<j} \int_{A_k} \omega_{0,k}(s)K(E_k D\phi(0,t) \cdot s)ds \label{err}\ee

We split the sum on the left into two parts: $k \geq j$ and $k < j$. For $k \geq j$, the sum is equal to

\be\begin{split}\label{k>j}E_j \left( \sum_{k\geq j} \int_{A_k} \omega_{0,k}(s) K(\phi(s,t)-\phi(0,t))ds \right) &\lesssim  \sum_{k\geq j} \dnorm{K}_{L^\infty(\phi(A_k,t))} \int_{A_k} E_j \omega_{0,k}(s) ds \\
& \lesssim (\log N)^2.  \end{split}
\ee

One factor of $\log N$ comes from integrating $K$, and the other comes from (\ref{conservation}) and the fact that $E_j \omega_{0,k} = 0$ for $k>j+\log N +2$. To control the rest of the error (\ref{err}), where the first sum is over $k<j$, we have

\be \label{k<j}\sum_{k<j} \int_{A_k} \omega_{0,k}(s)\big(K(\phi(s,t)-\phi(0,t))-K(E_kD\phi(0,t)\cdot s)\big)ds.\ee

By Lemma \ref{diff}, we have $\norm{\phi(s,t)-\phi(0,t) - E_k D\phi(0,t) \cdot s} \lesssim 2^{-k}N^{-\frac{9}{10}}$ when $\norm{s} \sim 2^{-k}$. Further, by (\ref{Dphismall}) and (\ref{phismall}), we may choose $C$ so that if $x = \phi(s,t)-\phi(0,t) $, $y= E_k D\phi (0,t) \cdot s$ and $\eps = \frac{1}{50} - \frac{1}{500}$, we have

$$2^{-k}N^{-\eps} \lesssim \norm{x},\norm{y} \lesssim 2^{-k}N^{\eps}$$

for times $t\lesssim C(\log N)/N$. Then we have the bound

\begin{align*}K_{12}(x)-K_{12}(y) & \lesssim 2^{4k}N^{4\eps}(x_1x_2 - y_1y_2)\\
&= 2^{4k}N^{4\eps}(x_1(x_2-y_2)+y_2(x_1-y_1))\\
&\lesssim N^{5\eps}2^{3k} \max_i \{\norm{x_i-y_i}\} \\
&\lesssim 2^{2k}N^{5\eps-\frac{9}{10}}\\
&\lesssim 2^{2k}N^{-\frac{4}{5}-\frac{1}{100}}
\end{align*}

and similarly for $K_{11}$. We can then estimate the sum (\ref{k<j}) by

$$N^{-\frac{4}{5}-\frac{1}{100}} \sum_{k<j} \dnorm{\omega_{0,k}}_{L^\infty}ds \lesssim N^{\frac{1}{5}-\frac{1}{100}} \log N \lesssim N^\frac{1}{5}$$

and with this we have the estimate (\ref{bigode}).\\

We will now apply a version of Lemma \ref{ODE}. Assume for contradiction that the estimate $\norm{h_k - E_kD\phi} = O(N^{-\frac{7}{10}})$ fails for the first time at time $t_0 < C(\log N)/N$ and at scale $j$. So, for $k<j$, the estimate holds. Therefore we have

\begin{align*} \frac{d(E_jD\phi - h_j)}{dt} &= \left( \sum_{k<j} (\nabla u)_{k,E_kD\phi} \right)E_kD\phi - \left( \sum_{k<j} (\nabla u)_{k,h_k} \right) h_j + O(N^\frac{1}{5})E_jD\phi \\
&= \left( \sum_{k<j} (\nabla u)_{k,h_k} \right) (E_kD\phi - h_j) + \left( \sum_{k<j} (\nabla u)_{k,E_kD\phi} - (\nabla u)_{k,h_k} \right) E_jD\phi + O(N^\frac{1}{5})E_jD\phi\\
& \lesssim \left( \sum_{k<j} (\nabla u)_{k,h_k} \right) (E_jD\phi - h_j) + O(N^\frac{1}{5})E_jD\phi
\end{align*}

where, for the last line, we used our assumption on small scales $k<j$ and the estimates on the Biot-Savart kernels $K_{1i}$. Note that, at time $t=0$, the difference $E_jD\phi - h_j = 0$. Suppose that $T$ is the first time such that $E_jD\phi-h_j = N^{-\frac{4}{5}}$. If $t\leq \min\{\frac{1}{N},T\}$, we have

$$\frac{d(E_jD\phi-h_j)}{dt} \lesssim N^\frac{1}{5} \tx{since} N^\frac{1}{5} E_jD\phi = O(1)$$

and it follows that $T = O(\frac{1}{N})$. For times $t$ such that $T \leq t \leq t_0 < C(\log N)/N$, the first term dominates and

$$E_jD\phi - h_j = O(N^{-\frac{4}{5}} \exp \big(tO(N))\big) = O(N^{-\frac{7}{10}})$$

where the last equality comes from our choice of $C$, since $t_0 < C(\log N)/N \leq (\log N^\frac{1}{10})/N$. Thus, the assumption that the estimate breaks down at scale $j$ and at time $t_0 < C(\log N)/N$ was false, and hence it holds for all $j$ and $t\leq C(\log N)/N$, proving the claim.

\end{proof}

\bigskip
\tiny

\textsc{N. H. KATZ, DEPARTMENT OF MATHEMATICS, CALIFORNIA INSTITUTE OF TECHNOLOGY, PASADENA, CA}

{\it nets@caltech.edu}

\bigskip

\textsc{A. TAPAY, DEPARTMENT OF MATHEMATICS, INDIANA UNIVERSITY, BLOOMINGTON, IN}

{\it atapay@indiana.edu}

\end{document}